\documentclass[a4paper,14pt]{article}
\usepackage[utf8]{inputenc}
\usepackage{amsmath,amsfonts,amssymb,amsthm,epsfig,epstopdf,titling,url,array}

\theoremstyle{plain}
\newtheorem{thm}{Theorem}[section]
\newtheorem{lem}[thm]{Lemma}
\newtheorem{prop}[thm]{Proposition}
\newtheorem{cor}[thm]{Corolary}
\newtheorem{defn}{Definition}[section]
\newtheorem{rem}{Remark}[section]

\newtheorem{claim}{\bf Theorem}

\def\R{\mathbb{R}}
\def\sph{\mathbb{S}}
\def\pc{\mathbb{C}}

\def\D{\mathbb{D}}

\title{A Picard type theorem for Hyperbolic Gauss Map of CMC-1 
Surfaces in Hyperbolic 3-space and de Sitter 3-space}
\author{N. A. de Andrade and L. P. Jorge}

\begin{document}

\maketitle

\begin{abstract}
In a recent paper Jorge and Mercuri proved that the image of 
Gauss map of a complete non flat minimal surfaces in $\mathbb R^3$ with finite 
total curvature omits at
most $2$ points. In this work we follow their idea and prove a
similar result for CMC-1 with finite total curvature in $\mathbb{H}^{3}$
and CMC-1 faces with finite type and regular ends in $\mathbb{S}_{1}^{3}$.  
\end{abstract}

\section{Introduction}

\ \ \  A classical problem in theory of minimal surfaces of  
$\mathbb{R}^{3}$ is to determine the number of the missing points on the 
image of the Gauss map, in other words, if $M$ is a complete minimal surface 
and $G: M\longrightarrow \mathbb{S}^{2}$ is the Gauss map of $M$, how many 
points  in $\mathbb{S}^{2} \backslash G(M)$ are there? Fujimoto showed that if 
$M$ is not the plane, then $\sharp (\mathbb{S}^{2}\backslash G(M)) \leqslant 4$ 
and is exactly $4$ for the Scherk's surface. Osserman, studied this problem 
under the hypothesis that the total curvature is finite. In this case, he 
showed that $\sharp (\mathbb{S}^{2}\backslash G(M)) \leqslant 3$.
 Recently, Jorge and Mercuri \cite{kn:J-M} improved Osserman's estimate in the  
following theorem
\begin{thm}[Jorge, Mercuri]\label{thm:thJM}
 If the Gauss map of a complete minimal surface of $\mathbb{R}^3$ with finite 
total curvature misses $3$ or more points of the sphere then it is a plane.
\end{thm}
This estimative is sharp since catenoid's Gauss map omits exactly 2 points. 
So the problem is fully answered for this class of surfaces.
Theorem 1.1 follows from analysis of immersions with Gauss map missing three 
points.

 A natural question arises: Which other class of surfaces have the same 
conformal type of the complex plane and Gauss map as meromorphic functions?

 As first choice we have the Bryant's Surfaces. Those surfaces share many 
properties with minimal surfaces of $\mathbb{R}^3$ and act like an analogue 
of minimal surfaces in hyperbolic 3-space. Indeed, Collin, Hauswirth and 
Rosenberg \cite{kn:C-H-R-1} proved an analogue of Osserman's result for Bryant's 
surfaces, 
that is, the hyperbolic Gauss map of a complete Bryant's surface with finite 
total curvature misses at most 3 points.
 
 Kawakami \cite{kn:K} proved partial results for other classes of surfaces. He 
studied 
the hyperbolic Gauss map of algebraic Bryant Surfaces and CMC-1 faces in de 
Sitter 3-space. Algebraic Bryant Surfaces are Bryant Surfaces with finite total 
curvature dual (dual in the Umehara and Yamada's sense). CMC-1 faces are 
spacelike CMC-1 surfaces in de Sitter space with a specific kind of singularities. 
Kawakami showed that hyperbolic Gauss map of a complete algebraic Bryant's 
Surfaces and of a complete CMC-1 face can miss 3 points at most, otherwise it is 
constant.    
 
In this paper, we use Jorge and Mercuri's ideas to prove the following 
theorem:
\begin{claim}\label{th}
Let $M$ be one of the following:
\begin{enumerate}
 \item[(1)] a  complete Bryant surface with finite total 
curvature or,
\item[(2)] a complete algebraic Bryant surface,
\item[(3)]  a complete CMC-1 face of finite type with elliptic ends.
\end{enumerate}
  Then the hyperbolic Gauss map $G$ of $M$ 
omit at most 2 points, otherwise $G$ is constant. Further, in case of constant 
hyperbolic Gauss map $M$ is a  horosphere for items (1) and (2) and a 
horosphere space like type surface for item (3).  
\end{claim}
All results in this theorem are sharp. 
 
A surface is {\em parabolic} if positive harmonic functions are constant.  
 By \cite{kn:J-Meeks} the volume growth of geodesics ball $B_r(x_o)$ of a 
complete 
minimal surface into $\R^3$ with finite total curvature is of type $r^2$. In 
\S5 we study the decay of curvature of those surfaces of theorem 1 and 
concluded that all of them have $\text{vol}_M(B_r)\leq cr^2,\ r\geq r_0$ form 
some constant $c_0$ and fixed $r_0>0.$ This implies next result.

\begin{claim}\label{th2}
 Let $M$ be one complete surface of type:
\begin{enumerate}
 \item[(1)] minimal surface with finite total curvature  into $\R^3$;
 \item[(2)] Algebraic Bryant surface with finite total curvature;
 \item[(3)] Algebraic Bryant surface endowed with the dual metric;
 \item[(4)] CMC-1 Face surface of finite type, elliptic ends endowed with the lift metric.
\end{enumerate}
Then $M$ is parabolic. If $M$ has a K\"ahler structure then all bounded 
holomorphic map $g\colon M\to\pc$ is constant. 
\end{claim}
\begin{rem}
 All surfaces in items (1),\ (2),\ (3),\ (4), are isometric to a complete 
minimal surface $M$ into $\R^3$ with finite total curvature. Then they have a 
natural K\"ahler structure. If the Gauss map $G$ of $M$ miss at least one point 
$y\in\sph^2$  we can rotated $M$ inside $\R^3$ to making $y=(0,\ 0,\ 1).$ If 
$g\colon M\to\pc$ is the stereographic projection of $G$ then $g$ is 
holomorphic. Hence all surfaces described in the theorem \ref{th2} are K\"ahler. 
 Unfortunately, if $g(M)=\pc\setminus\{a,\ b\}$ we can not do a lifting of $g$ 
to the unit disk $\D$ by a modular function ot get a easy proof of theorem 1.1 
unless $M$ is simply connected.      
\end{rem}


\section{Hyperbolic Space}

 
\ \ \  Let $\mathbb{L}^{4}$ be Lorentz-Minkowski 4-space with Lorentzian metric
 
\begin{equation}\label{HS1}
  \langle (x_{0}, x_{1}, x_{2}, x_{3}), (y_{0}, y_{1}, y_{2}, y_{3}) \rangle = 
-x_{0}y_{0} + x_{1}y_{1} + x_{2}y_{2} + x_{3}y_{3}
\end{equation}
 
 Define the hyperbolic $3$-space 
 
\begin{eqnarray*}
 \mathbb{H}^{3} = \{ v \in \mathbb{L}^{4} | \langle v, v\rangle = -1, x_{0}(v) 
> 0  \} 
\end{eqnarray*}
  
with the induced metric from $\mathbb{L}^{4}$. Give to $\mathbb{H}^{3}$ the 
orientation which $v_{1}, v_{2},v_{3}$ forms a oriented base of 
$T_{v}\mathbb{H}^{3}$ if, and only if, $v, v_{1}, v_{2},v_{3}$ forms a base of 
$\mathbb{L}^{4}$. We have that $ \mathbb{H}^{3}$ is a Riemannian manifold 
3-dimensional simply connected with constant mean curvature $-1$. 
 
 The space $\mathbb{H}^{3}$ is not compact, but can be compacted by adding a 
"sphere at infinity" $\mathbb{S}_{\infty}^{2}$ in such way that the rigid 
motions of $\mathbb{H}^{3}$ extend to homeomorphisms of 
$\overline{\mathbb{H}^{3}} = \mathbb{S}_{\infty}^{2} \cup \mathbb{H}^{3}$.

 Identify $ \mathbb{L}^{4}$ with the set $Herm(2) = \{X^{*} = X \}$ of $2 
\times 2$ hermitian matrices in the follow way
 
 \begin{equation}\label{HS2}
  (x_{0}, x_{1}, x_{2}, x_{3}) \mapsto \left[\begin{array}{c c }
                                            x_{0} + x_{3}&x_{1} + ix_{2}\\      
                                     
                                            x_{1} - ix_{2}&x_{0} - x_{3}
                                            \end{array}\right]
 \end{equation}

 where $i = \sqrt{-1}$.
 
 With this identification, we can see $ \mathbb{H}^{3}$ as $$ \mathbb{H}^{3} = 
\{ XX^{*} | X \in SL(2, \ \mathbb{C}) \} \ \ \ \ \ X^{*} = \overline{X}^{t}$$ 
with the metric 
 $$\langle X, Y \rangle = - \frac{1}{2}tr(X,\tilde{Y}), \ \langle X, X\rangle = 
-det(X)$$ where $\tilde{Y}$ is the cofactor of $Y$. 
 $PSL(2, \mathbb{C}) = SL(2, \mathbb{C})/\{\pm id\}$ acts isometrically on 
$\mathbb{H}^{3}$  by 
 
 \begin{equation}\label{HS3}
  X \mapsto YXY^{*} , \ X \in \mathbb{H}^{3} \, \ Y \in PSL(2, \mathbb{C})
 \end{equation}
 
 Note that the map $PSL(2, \mathbb{C}) \longrightarrow Herm(2)$ such that $F 
\mapsto FF^{*}$ take values in $\mathbb{H}^{3}$.    

\subsection{Bryant Surfaces}

\ \ \ Surfaces of constant mean curvature (CMC) are objects of great interest 
because they are critic points of functional area with respect to volume 
preserving variations that fix their boundaries. Minimal surfaces are a special 
class of CMC surfaces that are critical for all variations, not just volume 
preserving ones. Beside that, minimal surfaces can be described by a pair of 
holomorphic functions, called Weierstrass data. Many properties of minimal 
surfaces can be obtained using this Weierstrass data.

In 1970, Lawson described a correspondence between minimal surfaces and CMC-1 
in $\mathbb{H}^{3}$, the "Lawson's correspondence". So, as the minimal surfaces 
are described by a pair of holomorphic function, the Weierstrass data, one may 
ask if CMC-1 surfaces in $\mathbb{H}^{3}$ can be described in a similar way. 
Such question was answered by Robert Bryant \cite{kn:B} in 1987. He showed that this 
class of surfaces can be described by a pair of holomorphic functions, known as 
the \textit{Bryant data}. 
 
\begin{thm}[Bryant Representation]\label{thBriantRep}
 Let $M$ be a Riemann's surface and $F: M \longrightarrow SL(2, \mathbb{C})$ 
conformal immersion such that$det(F^{-1}dF) = 0$. Let $\phi: M \longrightarrow 
\mathbb{H}^{3}$ be such that $\phi = FF^{*}$. So $\phi$ is a immersion from $M$ 
to $\mathbb{H}^{3}$ with constant mean curvature 1. Conversely, if $\phi: M 
\longrightarrow \mathbb{H}^{3}$ is a immersion with constant mean curvature 1, 
there is a holomorphic lift $\phi$ to universal recovery $\tilde{F}: \tilde{M} 
\longrightarrow SL(2, \mathbb{C})$, such that $det(F^{-1}dF) = 0$
 and $\phi = \tilde{F}\tilde{F}^{*}$.
 
\end{thm}

Because of this work, the CMC-1 surfaces in hyperbolic 3-space became known as 
\textit{Bryant Surfaces}.

\begin{rem}
A holomorphic map $F: M \longrightarrow SL(2, \mathbb{C})$, such that 
$det(F^{-1}dF) = 0$ is called \textit{holomorphic null immersion}.
\end{rem}

\begin{rem}
Let $F: M \longrightarrow SL(2, \mathbb{C})$ be a map such that  $det(F^{-1}dF) 
= 0$. Locally, we can write:

  \begin{equation}\label{BR1}
  F^{-1}dF =  \left[\begin{array}{c c }
                     g & -g^{2}\\                                           
                     1 & -g
                     \end{array}\right] \omega
 \end{equation}
 
Where $g$ is a meromorphic function and $\omega$ is a holomorphic 1-form. 
Indeed, write
 
   \begin{equation}\label{BR2}
  F(z) =  \left[\begin{array}{c c }
                     A(z) & B(z)\\                                           
                     C(z) & D(z)
                     \end{array}\right] \in SL(2, \mathbb{C})
 \end{equation}
 
 Where $z$ is conformal local coordinate and $A, B, C, D$ are holomorphic. So, 
take $$g = - \frac{dB}{dA}, \ \omega = AdC - CdA$$
  
\end{rem}

 As $F$ is holomorphic, the forms $\omega$ e $g^{2}\omega$ are also 
holomorphic, and  $F^{-1}dF$ never vanish, by the fact that $F$ is a immersion. 
Therefore poles of $g$ are zeros of $\omega$, and a pole of order $k$ of $g$is a 
zero of order $2k$ of $\omega$. The pair $(g, \omega)$ is the Bryant data 
associated to $F$.
 
\begin{rem}
 We can write the holomorphic 1-form $\omega$, in local coordinates, as $\omega 
= f(z)dz$. So, we can write the Bryant data as $(g, f)$.
\end{rem} 
 
 Let $F$ be  a immersion of $M$ in $SL(2, \mathbb{C})$ such that 
 
\begin{equation}\label{BR3}
 F^{-1}dF =  \left[\begin{array}{c c }
                     g & -g^{2}\\                                           
                     1 & g
                     \end{array}\right] \omega 
\end{equation}

By Bryant's representation theorem and Lawson's correspondence, if $(g, f)$ is 
a Bryant data of a Bryant surface
so the metric and Gaussian curvature can be described using $g$ and $F$ by. 

$$ds^{2} = (1 + |g|^{2})^{2}|f|^{2}dz^{2}$$

$$\kappa(z) = \frac{-4|g'|^2}{|f|^2(1+|g|^2)^4} $$

Note that these are the same expressions for the Weierstrass data and minimal 
surfaces. 

\begin{rem}
 Note That the Weierstrass representation can be obtained as the limit of 
Bryant's theorem by collapsing the Lie group $SL(2, \mathbb{C})$ in the 
abelian group $\mathbb{C}^{3}$. 
\end{rem}

\subsection{Hyperbolic Gauss Map}

In 1986, Epstein introduced the hyperbolic Gauss map while he studied immersed 
surfaces in the hyperbolic 3-space in Poincaré model. Later, Bryant \cite{kn:B} 
reintroduced the concept in the hyperboloid of Minkowski space model, in the 
following way:

\begin{defn}\label{HGMdef}
 Let $M$ be a immersed surface in the hyperbolic 3-space. Define the 
\textit{hyperbolic Gauss Map} 
 $G:M \longrightarrow \mathbb{S}_{\infty}^{2}$ in the following way:
Let $z \in M$, take $\gamma$ normal oriented geodesic starting from $z$. Define 
$G(z)$ as the point where $\gamma$
intercepts the ideal boundary $\mathbb{S}_{\infty}^{2}$ of $\mathbb{H}^{3}$. In 
other words $$G(z) = \lim\limits_{t \rightarrow \infty} \gamma (t)$$
\end{defn}

For Bryant surfaces we have the follow characterization. Let $M$ be a Bryant 
surface and $F:M \longrightarrow SL(2, \mathbb{C})$ the conformal immersion 
given by theorem $2.1$. Denote 

\begin{eqnarray}\label{HGM1}
 F(z) = \left[\begin{array}{c c}
                     A(z) & B(z)\\                                           
                     B(z) & D(z)
                     \end{array}\right]
\end{eqnarray}

then,

\begin{eqnarray}
 G(z) = \frac{dA}{dC} = \frac{dB}{dD}\label{HGM2}
\end{eqnarray}

Following the terminology introduced by Umehara and Yamada,  $g(z) =  - 
\frac{dB}{dA}$ is called \textit{Secondary Gauss Map}. Note that the secondary 
Gauss map is the map $g$ from Bryant data $(g, \omega)$.
the motivation for the name comes from the fact that in a minimal surface 
immersed in $\mathbb{R}^{3}$ with  Weierstrass data $(\tilde{g}, \tilde{f})$ 
the map $\tilde{g}$ can be seen as the composition of Gauss map with 
stereographic projection. So, the map $\tilde{g}$ has 2 fundamentals roles: 

\begin{enumerate}
 \item[(i)] Describe the metric. 
 \item[(ii)] Describe the stereographic projection of unit normal vector. 
\end{enumerate}

But in the case of Bryant data $(g, f)$, the roles $(i)$ e $(ii)$ are played by 
2 different maps, $g$ plays the role $(i)$ and $G$ plays the role $(ii)$. 
Therefore, we can think that the Bryant representation has 2 Gauss Maps, the 
Hyperbolic Gauss Map and the Secondary Gauss Map.

Umehara e Yamada showed the following relationship between the Hyperbolic 
Gauss Map and Secondary Gauss Map.

\begin{eqnarray*}
 S(g) - S(G) = 2Q
\end{eqnarray*}

where

\begin{eqnarray*}
 S(g) &=& S_{z}(g)dz \\
 S_{z}(g) &=& \left( \frac{g''}{g'} \right) ' - \frac{1}{2} \left( 
\frac{g''}{g'} \right) ^{2} \\
 Q &=& \omega dg
\end{eqnarray*}

$S_{z}(g)$ is the \textit{Schwarzian derivate} and $Q$ is the \textit{Hopf's 
differential}.

\subsection{Finite Total Curvature}

\ The basic tools in the modern theory of minimal immersions $\phi: M^{2} 
\longrightarrow \mathbb{R}^{3}$ was establish by Osserman \cite{kn:O}. In this 
paper he shows  the following  results about complete minimal surfaces $M$ of 
$\mathbb{R}^3$ with finite total curvature:
\begin{enumerate}
 \item There is a compact Riemann surface $\overline{M}$ of finite genus $\mu$ 
and  a finite set $E=\{p_1,\cdots,p_k\}\subset\overline{M}$ such that $M$ is 
conformal to $\overline{M}\setminus E.$
\item The Gauss map $G\colon M\to\mathbb{S}^2$ extends to a conformal branched 
covering map $G\colon \overline{M}\to\mathbb{S}^2$.
\item $\sharp\big(\mathbb{S}^2\setminus G(M)\big)\leq 3$ unless $M$ is a flat 
plane.
\end{enumerate}
The points $p_{i} \in E$ or some times one neighborhood of them are known as 
\textit{ends of $M$.} In the hyperbolic context those items have similar 
results with fill exception. 

\begin{thm}[Bryant, \cite{kn:B}]\label{thBryant}
 Let $M$ a complete Bryant surface with finite total curvature immersed in 
$\mathbb{H}^{3}$. Then there is a compact Riemann surface $\overline{M}$ and a 
finite set $E = \{p_{1}, \ ... \ , p_{n} \} \subset \overline{M}$ such that $M$ 
is conformal to 
 $\overline{M} \setminus E$.\end{thm}
\begin{thm}[Collin, Hauswirth, Rosenberg, \cite{kn:C-H-R-1}]\label{thCHR1}
 Let $M$ be a immersed complete Bryant surface with finite total curvature in 
$\mathbb{H}^{3}$. Then the hyperbolic Gauss map $G$ of $M$ omit at most 
$3$ points unless $G$ is constant and $M$ is a horosphere.
\end{thm}

 The hyperbolic Gauss map is holomorphic but does not necessarily has 
meromorphic extension to the ends. In fact some ends could be an essential 
singularity of the Gauss map and motivated the next definition. 

\begin{defn}
Let $M$ be a complete Bryant surface with finite total curvature and $E = 
\{p_{1}, \ ... \ , p_{n} \}$ the ends of $M$. We say that an end $p_{i}$ is 
\textit{regular} if the hyperbolic Gauss map does extend meromorphically to 
$p_{i}$. Otherwise the end $p_{i}$  is called \textit{irregular}.
\end{defn}

The existence of irregular ends is the first big difference between minimal 
surfaces and Bryant surfaces.

\subsection{Dual Surfaces}

\ \ \ It is well known that a minimal surface $M$ with finite total 
curvature immersed in $\mathbb{R}^{3}$ satisfies the Osserman's inequality

\begin{eqnarray}
 \frac{1}{2 \pi} \int _{M} KdA \leq (\chi (M) - n)\label{SD1}
\end{eqnarray}

where $K$ is the Gaussian curvature of $M$ and $n$ is the number of ends of the 
surface. For Bryant surfaces with finite total curvature, there is no analogue  
relation. In fact, Umehara and Yamada \cite{kn:U-Y-2} showed that Bryant surfaces satisfy 
only the Cohn-Vossen inequality

\begin{eqnarray}\label{SD2}
  \frac{1}{2 \pi} \int _{M} KdA < \chi (M)
\end{eqnarray}

\ \ \ Motivated by this fact, Umehara and Yamada \cite{kn:U-Y-2} introduced the concept of 
dual surface, in a way that if $\phi : M^{2} \longrightarrow \mathbb{H}^{3}$ is 
a CMC-1 immersion and $\phi^{\sharp} : \tilde{M}^{2} \longrightarrow 
\mathbb{H}^{3}$ is the dual immersion, then $M$ satisfies an analogue of 
Osserman's inequality, but in terms of dual surface. 

\ \ \ Let $\phi: M \longrightarrow \mathbb{H}^{3}$ be a complete CMC-1 
immersion with finite total curvature. Let $G$ be the hyperbolic Gauss map, 
$(g,\omega)$ its Bryant data and $Q = \omega dg$ the Hopf differential of 
$\phi$. 

\begin{defn}\label{SDdef}
 Define the dual CMC-1 immersion $\phi^{\sharp}: \tilde{M} \longrightarrow 
\mathbb{H}^{3}$ associated to Bryant data $(g, \omega)$ of $\phi$ by
 
 \begin{eqnarray*}
\phi^{\sharp} = (F^{\sharp})(F^{\sharp})^{*}  
 \end{eqnarray*}

where $F$ satisfies $(4)$, $F^{\sharp}$ is the inverse of the matrix $F$ and 
$\tilde{M}$ is the universal cover of $M$.
\end{defn}

\begin{rem}
Observe that $\phi^{\sharp}$ is not necessarily \textit{single-valued} in $M$. 
But is easy see that $\phi^{\sharp}$ is \textit{single-valued} in $M$ if, and 
only if, $g$ is \textit{single-valued} in $M$.
\end{rem}

Let $(g^{\sharp}, \omega^{\sharp})$ be the pair defined by

\begin{equation}\label{SD3}
  (F^{\sharp})^{-1}dF^{\sharp} =  \left[\begin{array}{c c }
                     g^{\sharp} & -(g^{\sharp})^{2}\\                           
                
                     1 & -g^{\sharp}
                     \end{array}\right] \omega^{\sharp}
 \end{equation}

thus, taking $(\phi^{\sharp})^{\sharp}$, we get a map 
$(F^{\sharp})^{\sharp}: \tilde{\tilde{M}} \longrightarrow \mathbb{H}^{3}$ such 
that, $(F^{\sharp})^{\sharp} = (F^{-1})^{-1} = F$, this shows that 
$(\phi^{\sharp})^{\sharp} = \phi$.

Now we show a important result due to Umehara and Yamada \cite{kn:U-Y-2} that shows the 
relationship between the Bryant data and its dual data.

\begin{prop}[Umehara, Yamada]\label{thUY}
 Let $\phi^{\sharp}$ be the dual immersion of the CMC-1 immersion $\phi$ of a 
surface $M$ in a hyperbolic 3-space. Let $(g, \omega)$ be the Bryant data of 
$\phi$. So, $\phi^{\sharp}$ is a CMC-1 immersion of $M$ in the hyperbolic 
3-space, with hyperbolic Gauss map $G^{\sharp}$, Bryant data $(g^{\sharp}, 
\omega^{\sharp})$ and Hopf differential $Q^{\sharp}$  $\phi^{\sharp}$ given 
by:
 
 \begin{eqnarray}
  G^{\sharp} = g \ , \ g^{\sharp} = G \ , \ \omega^{\sharp} = - \frac{Q}{dG} \ 
, \ Q^{\sharp} = - Q
 \end{eqnarray}

\end{prop}

\begin{rem}
 Observe that the dual surface changes the hyperbolic Gauss map with the 
secondary Gauss map of $M$.
\end{rem}

If $\phi$ is a CMC-1 immersion in hyperbolic 3-space and $\phi^{\sharp}$ its 
dual immersion. By the above theorem, $\phi^{\sharp}$ is a CMC-1 immersion 
also, so it admits a Bryant data $(g^{\sharp}, \omega^{\sharp})$. Thus the 
metric of the dual immersion is given by
 
\begin{eqnarray}
 ds^{\sharp 2} &=& (1 + |g^{\sharp}|^{2})^{2}\omega^{\sharp} 
\overline{\omega^{\sharp}} \\
               &=& (1 + |G|^{2})^{2}\frac{Q}{dG} \overline{(\frac{Q}{dG})}  
\end{eqnarray}

Such metric is called \textit{dual metric}.

\begin{rem}
 Observe that the metric $ds^{\sharp 2}$ is well defined and \textit{single 
valued} in $M$. So, make sense take $M$ with the metric $ds^{\sharp 2}$, such 
surface is a Bryant surface with Bryant data $(G, \frac{-Q}{dG})$.
\end{rem}

There is a relation between the metric $ds^{2}$ and the dual metric $ds^{\sharp 
2}$ \cite{kn:Z}

\begin{prop}[Yu]\label{thYu}
The dual metric $ds^{\sharp 2}$ is complete (resp. non degenerated) if, and 
only if, the metric $ds^{2}$ is complete (resp. non degenerated).
\end{prop}


\subsubsection{Algebraic Bryant Surfaces}

As the dual metric is well defined in $M$, we have the following definition:

\begin{defn}\label{SDdefDualTC}
Let $M$ be a Bryant surface. Define the \textit{dual total curvature} of $M$, 
as  

\begin{eqnarray}
K_{T}^{\sharp} = \int _{M} - K^{\sharp}dA^{\sharp} .\label{SD4}
\end{eqnarray}

Where $K^{\sharp}$ and $dA^{\sharp}$ are the Gaussian curvature and the area 
element of the surface $M$ with the dual metric.

\end{defn}

Observe that the dual total curvature is the area of $M$ with respect to the 
(singular) metric induced by Fubini-Study metric in $\mathbb{C}\mathbb{P}^{1}$.

\begin{defn}\label{SDdefFTC}
 Let $M$ be a Bryant surface. We say that $M$ is \textit{algebraic} if $M$ has 
finite dual total curvature.
\end{defn}

The algebraic Bryant surfaces satisfy the following theorem

\begin{thm}[Bryant, Huber, Z. Yu]\label{thBHYu}
Let $M be$ a algebraic Bryant surface. So:
 
 \begin{enumerate}
  \item[(i)] $M$ is biholomorphic to $\overline{M_{\gamma}} \setminus E$, where 
$\overline{M_{\gamma}}$ is a  closed surface of genus $\gamma$ and  $E \subset 
\overline{M_{\gamma}}$ is a finite set $E = \{p_{1} , ... , p_{m} \}$.
  \item[(ii)] the dual Bryant data $(G, \omega^{\sharp})$ extends 
meromorphically to $\overline{M_{\gamma}}$.
 \end{enumerate}

The points of set $E$ are the \textit{ends} of $M$. 
\end{thm}

Using this, Umehara and Yamada \cite{kn:U-Y-2} deduced an analogue to Osserman's 
inequality to dual surfaces. Explicitly they proved that

\begin{thm}[Umehara, Yamada]
 Let $M$ be a Riemann surface and $\phi:M \longrightarrow \mathbb{H}^{3}$ a 
CMC-1 complete conformal immersion with finite dual total curvature. Let 
$\phi^{\sharp} : {M}^{2} \longrightarrow \mathbb{H}^{3}$ be the dual immersion. 
then,

\begin{eqnarray}
 \frac{1}{2 \pi} \int _{M} K^{\sharp}dA^{\sharp} \leq (\chi (M) - n) \label{SD5}
\end{eqnarray}

Where $\tilde{K}$ and $dA^{\sharp}$ the dual Gaussian curvature and the dual 
area element, and $n$ is the number of ends of the original surface $M$. 
 
\end{thm}

Due o this properties, one may ask if there is a Picard Type theorem for the 
hyperbolic Gauss map of algebraic Bryant surfaces too. Indeed, some partial 
results were obtained by Kawakami \cite{kn:K}.

\section{De Sitter 3-Space}

\ \ \ In this section we study the \textit{CMC-1 faces}. This CMC-1 faces are 
spacelike CMC-1 surfaces in de Sitter 3-space with some kind of singularities. 
Such surfaces share a many properties with Bryant surfaces, in particular, they 
have an analogue for the Bryant representation. For CMC-1 faces, is possible to 
define a hyperbolic Gauss map in a similar way we did in surfaces immersed in 
hyperbolic 3-space, so we can estimate the number of omitted points in the 
image of this map. In this work we show a sharp estimative for this number.

Lets give a brief description of \textit{de Sitter} space.

Let $\mathbb{L}^{4}$ be the Lorentz-Minkowski 4-space with Lorentz metric
 
 \begin{equation}\label{DeS1}
  \langle (x_{0}, x_{1}, x_{2}, x_{3}), (y_{0}, y_{1}, y_{2}, y_{3}) \rangle = 
-x_{0}y_{0} + x_{1}y_{1} + x_{2}y_{2} + x_{3}y_{3}
 \end{equation}
 
Define the \textit{de Sitter} 3-space $\mathbb{S}_{1}^{3}$ as

 \begin{eqnarray*}
 \mathbb{S}_{1}^{3} = \{ v \in \mathbb{L}^{4} | \langle v, v\rangle = 1\} 
 \end{eqnarray*}

with the induced metric $\mathbb{L}^{4}$. $\mathbb{S}_{1}^{3}$ is a 
3-dimensional Lorentz manifold  simply connected with constant sectional 
curvature 1.

We identify $\mathbb{L}^{4}$ with the set of hermitian matrices $2 \times 2$  
$Herm(2) = \{X^{*} = X \}$ by

 \begin{equation}\label{DeS2}
  (x_{0}, x_{1}, x_{2}, x_{3}) \mapsto \left[\begin{array}{c c }
                                            x_{0} + x_{3}&x_{1} + ix_{2}\\      
                                     
                                            x_{1} - ix_{2}&x_{0} - x_{3}
                                            \end{array}\right]
 \end{equation}

 where $i = \sqrt{-1}$.
 
With this identification, we can define
 
 \begin{equation}\label{DeS3}
  e_{0} = \left[\begin{array}{c c }
                       1 & 0 \\                                           
                       0 & 1
                \end{array}\right] \ , \ e_{1} = \left[\begin{array}{c c }
                                                        0 & 1 \\                
                           
                                                        1 & 0
                                                       \end{array}\right]  \ , 
\ e_{2} = \left[\begin{array}{c c }
                                                                                
                      0 & i \\                                           
                                                                                
                      -i & 0
                                                                                
                \end{array}\right]   \ , \ e_{3} = \left[\begin{array}{c c }
                                                                                
                                                                  1 & 0 \\

                                                                  0 & -1
                                                                                
                                                           \end{array}\right]
 \end{equation}
 
 Thus, 
 
 \begin{eqnarray*}
  \mathbb{S}_{1}^{3} & = & \{X | X = X^{*} \ , \ det(X) = -1 \} \\
                     & = & \{Fe_{3}F^{*} | F \in SL(2, \mathbb{C}) \}
 \end{eqnarray*}

with the metric $$\langle X, Y \rangle = - \frac{1}{2} tr(Xe_{2}(Y^{t})e_{2}).$$
 
Where $X^{*} = \overline{X^{t}}$.
 
\begin{rem}
Observe that with above definition $\langle X, X \rangle = - detX$.
\end{rem}
 
\subsection{CMC-1 Face}
In this section, we define CMC-1 faces, enumerate some important results and at 
last we prove a Picard type theorem for this surfaces. 
 
\begin{defn}\label{CMC1Facedef}
A immersion $f: M^{2} \longrightarrow \mathbb{S}_{1}^{3}$ of a surface $M$ is 
called \textit{spacelike} if the induced metric in $M$ is positive definite.
\end{defn}

As an analogue for the Bryant surfaces, we have the following theorem \cite{kn:A-A}:
 
\begin{thm}[Aiyama-Akutagawa]\label{thAiAk}
Let $D$ be a simply connected domain in $\mathbb{C}$ and $z_{0} \in D$ a base 
point. Let 
  
  $$g:D \longrightarrow (\mathbb{C} \cup \{\infty \}) \setminus \{z \in 
\mathbb{C} | |z| \leq 1 \} $$
 
be a meromorphic function and $\omega$ a holomorphic 1-form in $D$ such that
 
 $$d\hat{s}^{2} = (1 + |g|^{2})^{2}\omega \overline{\omega}$$
 
is a Riemannian metric in $D$.
 
Take $F = (F_{ij}) : D \longrightarrow SL(2, \mathbb{C})$ holomorphic immersion 
such that $F(z_{0}) = e_{0}$ and
 
\begin{equation}\label{DeS4}
  F^{-1}dF =  \left[\begin{array}{c c }
                     g & -g^{2}\\                                           
                     1 & -g
                     \end{array}\right] \omega
\end{equation}

then $f: D \longrightarrow \mathbb{S}_{1}^{3}$ defined by
$$f = Fe_{3}F^{*}$$

is a conformal spacelike immersion, with constant mean curvature 1.

The induced metric $ds^{2}$ in $D$, satisfies

$$ds^{2} = (1 - |g|^{2})^{2}\omega \overline{\omega}$$
 
Conversely, every CMC-1 immersion of a simply connected surface has this form.
 \end{thm}

\begin{rem}
 The pair $(g, \omega)$ is called \textit{Aiyama data}.
\end{rem}

Following Umehara and Yamada definitions for CMC-1 immersions in 
$\mathbb{H}^{3}$, we define

\begin{defn}
Let $f:D \longrightarrow \mathbb{S}_{1}^{3}$ and $F = (F_{ij})$ as above. 
Define the \textit{hyperbolic Gauss Map} $G$ of $f$ as
 
 $$ G = \frac{dF_{11}}{dF_{21}} = \frac{dF_{21}}{dF_{22}}. $$
 
We define also the \textit{Hopf differential} $Q$ of immersion by 
 
 $$Q = \omega dg.$$
 
\end{defn}

\begin{rem}
The hyperbolic Gauss map has the following meaning. Let 
$\mathbb{S}_{\infty}^{2}$ be the pointing future ideal boundary ideal of 
$\mathbb{S}_{1}^{3}$. $\mathbb{S}_{\infty}^{2}$ can be identified with 
$\mathbb{C} \cup \{ \infty \}$. So, given $z \in D$, take $\gamma$ geodesic in 
$\mathbb{S}_{1}^{3}$ with initial velocity  as the unity normal vector of 
$f(D)$ in $f(z)$. Then $G(z)$ is the point where $\gamma$ intercepts the ideal 
boundary $\mathbb{S}_{\infty}^{2}$
\end{rem}

\begin{rem}
Observe that given $f: D \longrightarrow \mathbb{S}_{1}^{3}$ with $f = 
Fe_{3}F^{*}$, then $\hat{f}: D \longrightarrow \mathbb{H}^{3}$ given by 
$\hat{f} = FF^{*}$ is a conformal CMC-1 immersion, with metric $d\hat{s}^{2}$ 
and the same hyperbolic Gauss map $G$ and Hopf differential $Q$ of $f$.
\end{rem}

\begin{defn}
Let $M$ be a oriented surface. A smooth map $f: M \longrightarrow 
\mathbb{S}_{1}^{3}$ is called CMC-1 map if there is an open dense set $W 
\subset M$ such that $f|_{W}$ is a spacelike immersion. A point $p \in M$ is 
called a singular point of $f$ if the indexed metric $ds^{2}$ is degenerated in 
$p$.
\end{defn}

\begin{defn}
Let $f:M \longrightarrow \mathbb{S}_{1}^{3}$ be a CMC-1 map and $W \subset M$ 
an open dense set such that $f|_{W}$ is a CMC-1 immersion. A point $p \in M 
\setminus W$ is an \textit{admissible singular point} if:
 
 \begin{enumerate}
  \item[(1)] There is a map $\beta : U \cap W \longrightarrow \mathbb{R}^{+}$ 
of class $C^{1}$, where $U$ is a neighborhood of $p$, such that $\beta ds^{2}$ 
extends to a Riemannian metric in $U$ of class $C^{1}$.
  \item[(2)] $df(p) \neq 0$, that is, $df$ has rank 1 in $p$.
 \end{enumerate}

We say that a CMC-1 map $f$ is a CMC-1 face if all its singular points are 
admissible.
\end{defn}

\begin{prop}\label{DeS-prop1}
Let $M$ be a oriented surface and $f: M \longrightarrow \mathbb{S}_{1}^{3}$ 0a 
CMC-1 face where $W \subset M$ is the open dense such that $f|_{W}$ is a CMC-1 
immersion. Then there is a only one complex structure $J$ on $M$ such that
 
\begin{enumerate}
 \item[(1)] $f|_{W}$ is conformal with respect to $J$.
 \item[(2)] There is a immersion $F: \tilde{M} \longrightarrow SL(2, 
\mathbb{C})$ which is holomorphic with respect to $J$, such that 
 
 $$det(dF) = 0 \ e \ f o \varrho = Fe_{3}F^{*}$$
 
Where $\varrho: \tilde{M} \longrightarrow M$ is the universal cover of $M$.
\end{enumerate}

$F$ is called \textit{holomorphic null lift}.
 
\end{prop}

\begin{rem}
The holomorphic null lift is unique except of right multiplication for a 
constant matrix in $SU(1, 1)$.
\end{rem}

By the above proposition, given a CMC-1 face $f:M \longrightarrow 
\mathbb{S}_{1}^{3}$ , always exists a complex structure $J$ in $M$. 
Henceforward, $M$ will be treated as a Riemann surface with this complex 
structure.

\begin{prop}\label{DeS-prop2}
Let $M$ be a Riemann surface and $F: M \longrightarrow SL(2, \mathbb{C})$ a 
holomorphic null immersion. Assume that the symmetric $(0, 2)$-tensor 
$det[d(Fe_{3}F^{*})]$ is not identically zero. Then 
 
 $$f = Fe_{3}F^{*} : M \longrightarrow \mathbb{S}_{1}^{3}$$
 
is a CMC-1 face, and a point $p \in M$ is a singular point of $M$ if, and only 
if, $det[d(Fe_{3}F^{*})]_{p} = 0$. Beside that, $ - det[d(FF^{*})]$ is positive 
definite on $M$.
\end{prop}

\begin{rem}
The hypothesis of $det[d(Fe_{3}F^{*})] \neq 0$ is essential. In fact, take $F: 
M \longrightarrow \mathbb{S}_{1}^{3}$ holomorphic null immersion such that 
$$F(z) = \left[\begin{array}{c c }
                     z + 1 & - z\\                                           
                      z    & - z + 1
                     \end{array}\right] $$
                     
So, $f = Fe_{3}F^{*}$ degenerates in every point of $\mathbb{C}$, therefore 
does not provides a immersion. It comes from the fact that 
$det[d(Fe_{3}F^{*})]$ 
is identically zero.
\end{rem}

Using the above theorems, is possible to extend Aiyama-Akutagawa representation 
to CMC-1 faces which domain is not simply connected.

\begin{rem}
Let $F$ be a holomorphic null lift a CMC-1 face $f$ with Aiyama data $(g, 
\omega)$. Let $B \in SU(1, 1)$ be a constant matrix, that is
 
 $$ B =  \left[\begin{array}{c c }
                \overline{p} & - q\\                                           
               - \overline{} & p
               \end{array}\right] \in SU(1, 1) \ , \ p\overline{p} - q 
\overline{q} = 1 $$
               
Thus, $FB$ is a holomorphic null lift of $f$. The Aiyama data $(\hat{g}, 
\hat{\omega})$ corresponding to $(FB)^{-1}d(FB)$ is given by

$$\hat{g} = \frac{pg + q}{\overline{q}g + \overline{p}} \ e \ \hat{\omega} = 
(\overline{q}g + \overline{p})^{2} \omega. $$

2 Aiyama datas $(g, \omega)$ and $(\hat{g}, \hat{\omega})$ are called 
\textit{equivalents} if satisfy the above equality for some $B \in SU(1, 1)$. 
We call the equivalence class of Aiyama data $(g, \omega)$ as Aiyama data 
associated to $f$.
\end{rem}

\begin{rem}

Observe that, if $(g, \omega)$ e $(\hat{g}, \hat{\omega})$ are 2 equivalents 
Aiyama data, and if $G$ and $\hat{G}$ are the hyperbolic Gauss maps associated 
with this data, then

$$ G = \frac{dF_{11}}{dF_{21}} \ \ \ e \ \ \ \hat{G} = 
\frac{\overline{p}dF_{11} + \overline{q}dF_{12}}{\overline{p}dF_{21} + 
\overline{q}dF_{22}} $$

Therefore, $G = \hat{G}$, that is, the hyperbolic Gauss map does not depend of 
the choice of $F$ in the  class. 
\end{rem}

\subsection{CMC-1 Faces With Elliptics Ends}

\begin{defn}\label{CMC1Face-defEE}
 Let $M$ be a Riemann surface, and $f: M \longrightarrow \mathbb{S}_{1}^{3}$ a 
CMC-1 face. Let 
 $ds^{2} = f^{*}(ds_{\mathbb{S}_{1}^{3}}^{2})$. $f$ is  \textit{complete} 
(resp, of \textit{finite type}) if there is a compact set $C$ and a symmetric 
$(0, 2)$-tensor $T$ in $M$ such that 
 $T$ vanishes in $M \setminus C$ and $ds^{2} + T$ is a complete Riemannian 
metric (resp. has finite total curvature).
\end{defn}

\begin{rem}
 For CMC-1 immersions in $\mathbb{S}_{1}^{2}$, the Gaussian curvature $K$ is 
not negative, So the total curvature is the same as the absolute total 
curvature. However, for CMC-1 faces with singular points, the total curvature 
never is finite. 
\end{rem}

\begin{rem}
The universal cover of a  complete CMC-1 face (resp finite type) is not 
necessarily complete (resp finite type), Because the singular set may not be 
compact the universal cover.
\end{rem}

Let $f: M \longrightarrow \mathbb{S}_{1}^{3}$ be a complete CMC-1 face of 
finite type. So $(M, ds^{2} + T)$ is a  complete Riemannian surface with finite 
total curvature. Therefore, $M$ has finite topological type.

Let $\phi: \tilde{M} \longrightarrow M$ be the universal cover of $M$, and $F: 
\tilde{M} \longrightarrow SL(2, \mathbb{C})$ a holomorphic null lift of a CMC-1 
face $f:M \longrightarrow \mathbb{S}_{1}^{3}$. Fix a point $z_{0} \in M$. Let 
$\gamma: [0, 1] \longrightarrow M$ a loop such that $\gamma (0) = \gamma (1) = 
z_{0}$. So there is a unique deck transformation $\tau$ of $\tilde{M}$ 
associated to the homotopy class of $\gamma$. Define the \textit{monodromy 
representation} $\Phi_{\gamma}$ of $F$ by

$$Fo\tau = F\Phi_{\gamma}.$$

As $f$ is well defined in $M$, $\Phi_{\gamma} \in SU(1, 1)$ for every loop 
$\gamma$. Therefore, $\Phi_{\gamma}$ is conjugated to one of the following 
matrices

$$ E_{1} =  \left[\begin{array}{c c }
                e^{i\theta} & 0\\                                           
                          0 & e^{i\theta}
               \end{array}\right] \ ou \ E_{2} = \pm \left[\begin{array}{c c }
                                                           cosh s & senh s\\    
                                       
                                                           senh s & cosh s
                                                        \end{array}\right] \ ou 
\ E_{3} = \pm \left[\begin{array}{c c }
                                                                                
                        1 + i & 1\\                                           
                                                                                
                            1 & 1 - i
                                                                                
                     \end{array}\right] $$
                                                                                
                     De Sitter 3-Space
for $\theta \in [0, 2 \pi)$, $s \in \mathbb{R} \setminus \{0\}$.

\begin{defn}
 Let $f:M \longrightarrow \mathbb{S}_{1}^{3}$ be a complete CMC-1 face of 
finite type with holomorphic null lift $F$. A end of $f$ is called 
\textit{elliptic, hyperbolic or parabolic} if tis monodromy representation is 
conjugated to $E_{1}$, $E_{2}$ or $E_{3}$ in 
 $SU(1, 1)$ respectively.
\end{defn}

\begin{rem}
 As every matrix in  $SU(2, \mathbb{C})$ is conjugated to $E_{1}$ in $SU(2, 
\mathbb{C})$, CMC-1 immersions in 
 $\mathbb{H}^{3}$ have similar properties to CMC-1 faces with elliptic ends in 
$\mathbb{S}_{1}^{3}$.
\end{rem}

\begin{prop}\label{CMC1Face-prop1}
Let $V$ be a neighborhood of a end of $f$ and $f|_{V}$ a spacelike CMC-1 
immersion of finite total curvature, which is complete in the end. Suppose that 
the end is elliptic. So there is a holomorphic null lift $F: \tilde{V} 
\longrightarrow SL(2, \mathbb{C})$ of $f$ with associated Aiyama data $(g, 
\omega)$ such that $$d\hat{s}^{2}|_{V} = (1 + |g|^{2})^{2}|\omega|^{2}$$ is 
single valued in $V$. Besides, $d\hat{s}^{2}$ has complete total curvature and 
is complete at the end. 
\end{prop}

\begin{prop}\label{CMC1Face-prop2}
Let $f:M \longrightarrow \mathbb{S}_{1}^{3}$ be a complete CMC-1 face of finite 
type with elliptic ends. So, there is a compact Riemann surface $\overline{M}$ 
and a finite number of points $p_{1}, ... , p_{n} \in \overline{M}$, such that 
$M$ é biholomorphic to $\overline{M} \setminus \{p_{1}, ... , p_{n} \}$. 
Besides, the Hopf differential $Q$ of $f$ extends meromorphically to 
$\overline{M}$.
\end{prop}

Similarly to immersed Bryant surfaces in hyperbolic 3-space, the hyperbolic 
Gauss map does not extends necessarily to the ends. Because of this, we have 
the following definition

\begin{defn}\label{CMC1Face-defreg}
Let $f: M \longrightarrow \mathbb{S}_{1}^{3}$ a CMC-1 face. A end $p_{j}$ of 
$M$ is \textit{regular} if
 the hyperbolic Gauss map does extend meromorphically to $p_{j}$, otherwise is 
called \textit{irregular}.
\end{defn}

Let $f:M \longrightarrow \mathbb{S}_{1}^{3}$ be a CMC-1 face of finite type 
with elliptic ends. Thus $M = \overline{M} \setminus \{p_{1}, ... , p_{n} \}$. 
Let $G$ and $Q$ the hyperbolic Gauss map and the Hopf differential of $f$ 
respectively.

\begin{defn}\label{CMC1Face-defmetric}
 We call the metric 

 $$d\hat{s}^{\sharp 2} = (1 + 
|G|^{2})^{2}\left(\frac{Q}{dG}\right)\overline{\left(\frac{Q}{dG}\right)}$$

 of \textit{lift metric} of CMC-1 face $f$. Besides
 
 $$-(K_{d\hat{s}^{\sharp 2}})d\hat{s}^{\sharp 2}|_{V}  =  
\frac{4dGd\overline{G}}{(1 + |G|^{2})^{2}} $$
 
 \end{defn}

\begin{rem}
 Observe that $d\hat{s}^{\sharp 2}$ and $-(K_{d\hat{s}^{\sharp 
2}})d\hat{s}^{\sharp 2}$ are given in terms of $G$ and $Q$, and such functions 
are defined in $M$, thus $d\hat{s}^{\sharp 2}$ and $-(K_{d\hat{s}^{\sharp 
2}})d\hat{s}^{\sharp 2}$ are defined in $M$ also.
\end{rem}

 Fujimori \cite{kn:F} showed that

\begin{prop}\label{CMC1Face-thFuj}
 Let $f:M \longrightarrow \mathbb{S}_{1}^{3}$ be a CMC-1 face. Assume that each 
end of $f$ elliptic and regular. Thus, if $f$ is complete and of finite type, 
then the lift metric $d\hat{s}^{\sharp 2}$ is complete and of finite total 
curvature in $M$.
\end{prop}

\section{The proof of theorem 1}

\ \ \

\begin{proof}
 Suppose that $\phi:M \longrightarrow \mathbb{H}^{3}$ is a complete algebraic 
CMC-1 immersion with hyperbolic 
 Gauss map $G$ and Bryant data $(g, \omega)$. Then, take  $\phi^{\sharp}: M \longrightarrow \mathbb{H}^{3}$
its dual immersion with dual data $(G, \omega ^{\sharp})$ and let $\psi: M \longrightarrow \mathbb{R}^{3}$ be 
the minimal immersion cousin of the $\phi^{\sharp}$, with Weierstrass data $(G, \omega ^{\sharp})$. Note that
 $\psi$ is a complete minimal surface with finite total curvature, since $\phi^{\sharp}$ is complete
  and the total curvature of $\psi$ coincides with total dual curvature of 
$\psi^{\sharp}$, witch is finite since
  $\phi$ is algebraic. Besides, the Gauss map of $\psi$ coincides with the hyperbolic Gauss map $G$ of $\phi$.
 Then by theorem \ref{thm:thJM}, $G$ omits 2 points at most.

 Suppose now that $M$ is a complete Bryant surface with finite total curvature. 
Then $M$ is conformal 
to $\overline{M} \setminus E$, where $\overline{M}$ is a compact surface and $E$ is a finite set, the ends
of the surface. Suppose that $M$ has at least one irregular end $p \in E$. 
Then $p$ is a essential 
singularity of $G$, so by Big Picard's theorem, $G$ omits 2 points at most. 
Therefore we can suppose 
that all ends of $M$ are regular. If all ends are regular, then $M$ is a algebraic Bryant surface,
and therefore, $G$ can omits 2 points at most.

 Suppose now that $M$ is a a complete CMC-1 face of finite type with elliptic 
ends. Let $(g, \omega)$
be its Aiyama data. Define $\phi: M \longrightarrow \mathbb{H}^{3}$ the immersion wich has
$(g, \omega)$ as its Bryant data. Then, $\phi: M \longrightarrow \mathbb{H}^{3}$ is a complete
CMC-1 immersion. Note that the lift metric of the CMC-1 face $M$ coincides with the
dual metric of $\phi$. Then $\phi(M)$ is a algebraic Bryant surface, because $M$ has finite type.
Observe that the hyperbolic Gauss map $G$ of $M$ coincides with the hyperbolic Gauss map of 
$\phi: M \longrightarrow \mathbb{H}^{3}$. Therefore, $G$ can miss 2 points at most.
   
\end{proof}

\begin{rem}
This estimative is sharp since:

\begin{enumerate}
\item[(i)] the catenoid cousin is a complete Bryant surface with finite total curvature
witch hyperbolic Gauss map that omits exactly 2 points.
\item[(ii)] $\mathbb{C} \setminus \lbrace 0, 1 \rbrace$ is a complete algebraic Bryant 
surface witch hyperbolic Gauss map that omits exactly 2 points.
\item[(iii)]  the elliptical catenoid is a complete CMC-1 face of finite type with elliptic
 ends with hyperbolic Gauss map that omits exactly 2 points.
\end{enumerate}  
\end{rem}

\begin{cor}
 Let $M$ be a properly embedded Bryant surface of finite topological 
type. Then, the hyperbolic Gauss map is constant and $M$ is a horosphere or the 
image of hyperbolic Gauss map omits 2 points at most.
\end{cor}

\begin{proof}
 Collin, Hauswirth and Rosenberg showed in \cite{kn:C-H-R-2} that a
 properly embedded 
Bryant surface of finite topological type has finite total curvature, thus the 
result follows from the previous theorem.
\end{proof}

\section{Curvature estimate and volume growth.}

We begin with the following theorem

\begin{thm}\label{thm5.1}
Let $M$ be one complete surface of type:
\begin{enumerate}
\item[(1)] minimal surface with finite total curvature  into $\R^3$;
 \item[(2)] Algebraic Bryant surface with finite total curvature;
 \item[(3)] Algebraic Bryant surface endowed with the dual metric;
 \item[(4)] CMC-1 Face surface of finite type, elliptic ends endowed with the lift metric.
\end{enumerate}
If $\kappa\colon M\to\R$ is the curvature of $M$, there is a function $k: [0, 
\infty) 
\longrightarrow [0, \infty)$, at least $C^{2}$ such that,
 
 \begin{enumerate}
  \item[(i)] $0 \geq \kappa(p) \geq -k(r(p)),\quad \forall\ p\in M,\ 
r(p)=\text{dist}^M(p,p_0),\ p_0\ \text{fixed},$
  \item[(ii)] $\int_{0}^{\infty} sk(s)ds<\infty.$
 \end{enumerate}

\end{thm}

\begin{proof}
The Bryant surface with finite total curvature and regular ends is isometric to 
its cousin, that is, to one complete minimal surface into $\R^3$ with finite 
total curvature. 
The CMC-1 Face dual surface of finite type and elliptic ends is isometric to 
the dual of one
 algebraic Bryant surface (with finite total dual curvature) and both are 
isometric to one complete minimal surface into $\R^3$ with finite total 
curvature.

 We know that there is a compact Riemann surface $\overline{M}$ 
and a finite set $E = \{w_{1}, \ ... \ , w_{m} \}$ such that $M$ is conformal 
to $\overline{M} \smallsetminus E$, where $w_{i}$ are the ends of the surface. 
The Gauss map, for a minimal surface, and the hyperbolic Gauss map, for the
Bryant surface, $G$ of $M$ does extend meromorphically to the ends, and 
$G:\overline{M} \longrightarrow \mathbb{S}_{\infty}^{2}$ is a branched 
covering. 
Take $R_{0} > 0$ such that $\overline{M} \setminus B_{R_{0}}^{M}(o)$ is the 
finite 
union of the subends $E_{1}, ... , E_{m}$ . Bryant Proposition 4 \cite{kn:B} 
showed 
that
is possible parametrize each subend $E_{i}$ using the Bryant data $(g, f)$, 
in$\{z \in \mathbb{C} / 0 < 
|z| < r_{j} \ , \ \ r_{j} > 0\}$  by
 
 \begin{equation}
		\left\{ 
		\begin{array}{l l}
		g(z) = z^{1 + \beta (w_{i})} \hat{g}(z) \\
		f(z) = \frac{1}{z^{1 + I(w_{i})}} \hat{f}(z)
                \end{array} \right.
\end{equation}

 Where $\hat{g}$ and $\hat{f}$ are analytic, never vanish and extend 
meromorphically to $z=0$. Is well known that is possible to get a similar 
parametrization for minimal surfaces with finite total curvature.
 
 Let $\gamma_j=\partial E_{j}.$ For each $|z|<r_{j}$ let $\alpha_{j}$ be the 
path $re^{i \theta}$, $|z| \leqslant r \leqslant r_{j}$. We have 
 
\begin{eqnarray*}
 \text{dist}^M(\gamma_{j},z) &=& inf_{\alpha} \{ l(\alpha) \} \\
                           & \leq & l(\alpha) \\
                           &=& \int_{\alpha} |ds| \\
                           & \leq & \int_{r_j}^{|z|}|f|(1+|g|^2)|dz| \\
                           &=&  \int_{r_{j}}^{|z|} \frac{1}{|z|^{1 + 
I(w_{j})}}|\hat{f}(z)|(1 + |z|^{2 + 2 \beta(w_{j})}|\hat{g}(z)|^{2})|dz| \\
                           &=&  \int_{r_{j}}^{|z|} \frac{1}{|z|^{1 + 
I(w_{j})}}|\hat{f}(z)|(1 + |z|^{2 + 2 \beta(w_{j})}|\hat{g}(z)|^{2})|dz| \\
\end{eqnarray*}

Observe that $\hat{f}$ e $\hat{g}$ are bounded in $ 0 < |z| < r_{j}$. So, in 
this set, we have

\begin{eqnarray*}
dist_{M}(\gamma_{j},z) & \leq & \int_{r_{j}}^{|z|} \frac{1}{|z|^{1 + 
I(w_{j})}}|{\hat{f}}(z)|(1 + |z|^{2 + 2 \beta(w_{j})}|{\hat{g}}(z)|^{2})|dz| \\
                       & \leq & \int_{r_{j}}^{|z|} \frac{1}{|z|^{1 + I(w_{j})}} 
c_{1}(1 + |r_{j}|^{2 + + 2 \beta(w_{j})}c_{2}^{2})|dz| \\
                       & = & \int_{r_{j}}^{|z|} \frac{1}{|z|^{1 + 
I(w_{j})}}c_{3} |dz| \\
\end{eqnarray*}

Therefore

\begin{eqnarray}
 \text{dist}^M(\gamma_{j},z) \leq \frac{c_{3}}{|z|^{I(w_{j})}}.
\end{eqnarray}

Beside that, the Gaussian curvature $\kappa$ satisfies

\begin{eqnarray*}
 |\kappa(z)| &=& \frac{4|g'|^2}{|f|^2(1+|g|^2)^4} \\
             & \leq & \frac{4|z^{\beta(w_{j})}\hat{g}(z) + z^{1 + 
\beta(w_{j})}\hat{g}'(z)|^2}{|\frac{1}{z^{1 + I(w_{j})}}\hat{f}(z)|^2(1+|z^{1 + 
\beta(w_{j})}g_{1}(z)|^2)^4} 
\end{eqnarray*}

Observe that $1 \leq |1 + |g(z)|^{2}|$, therefore $\frac{1}{|1 + |g(z)|^{2}|} 
\leq 1$. Besides, $\hat{f}$, $\hat{g}$ and $\hat{g}'$ are bounded in $ 0 < |z| 
< r_{j}$. So,

\begin{eqnarray*}
 |\kappa(z)| & \leq & \frac{4|z^{\beta(w_{j})}\hat{g}(z) + z^{1 + 
\beta(w_{j})}\hat{g}'(z)|^2}{|\frac{1}{z^{1 + I(w_{j})}}\hat{f}(z)|^2(1+|z^{1 + 
\beta(w_{j})}\hat{g}(z)|^2)^4} \\
             & \leq & \frac{4(|z|^{\beta(w_{j})}|\hat{g}(z)| + |z|^{1 + 
\beta(w_{j})}|\hat{g}'(z)|)^2}{\frac{|\hat{f}(z)|^2}{|z|^{2 + 
2I(w_{j})}}(1+|z^{1 + \beta(w_{j})}\hat{g}(z)|^2)^4} \\
             & = & 4 \frac{|z|^{2 + 2I(w_{j})}(|z|^{\beta(w_{j})}|\hat{g}(z)| + 
|z|^{1 + \beta(w_{j})}|\hat{g}'(z)|)^2}{|\hat{f}(z)|^2(1+|z^{1 + 
\beta(w_{j})}\hat{g}(z)|^2)^4} \\
             & \leq & 4 \frac{|z|^{2 + 2I(w_{j})}(|z|^{\beta(w_{j})}\hat{c}_{1} 
+ |z|^{1 + \beta(w_{j}}\hat{c}_{2})}{\hat{c}_{3}} \\
             & = & \hat{c}_{4}|z|^{2 + 2I(w_{j}) + 2\beta(w_{j})}(1 + |z|)^{2} 
\\
             & = & \hat{c}_{4}|z|^{2 + 2I(w_{j}) + 2\beta(w_{j})}(1 + 
r_{j})^{2} \\
             & = & \hat{c}_{5}|z|^{2 + 2I(w_{j}) + 2\beta(w_{j})}
\end{eqnarray*}

Where $|\hat{g}(z)| \leq \hat{c}_{1}$, $|\hat{g}'(z)| \leq \hat{c}_{2}$ and 
$|\hat{3}(z)| \geq \hat{c}_{3}$ for $|z| \leq r_{j}$, and
$\hat{c}_{4} = 4.max\{\hat{c}_{1}, \hat{c}_{2} \}$.

Now, let $\rho(z)$ be the geodesic distance from $o \in M$ to $z$. Note that
 $\rho(z) \leq \frac{c_{3}}{|z|^{I(w_{j}}}$, so

\begin{eqnarray*}
 |z|^{I(w_{j})} & \leq & \frac{c_{3}}{\rho} 
\end{eqnarray*}

therefore,

\begin{eqnarray*}
 |z| & \leq & \frac{c_{3}}{\rho^{\frac{1}{I(w_{j})}}} 
\end{eqnarray*}

Replacing this in the Gaussian curvature inequality above

\begin{eqnarray*}
 |\kappa(z)| & \leq & c_{4}|z|^{2 + 2I(w_{j}) + 2\beta(w_{j})} \\
 - \kappa(z) & \leq & c_{4}|z|^{2 + 2I(w_{j}) + 2\beta(w_{j})} \\
   \kappa(z) & \geq & - c_{4}|z|^{2 + 2I(w_{j}) + 2\beta(w_{j})} \\
             & \geq & - c_{4} \left( \frac{c_{5}}{\rho^{\frac{1}{I(w_{j})}}} 
\right)^{2 + 2I(w_{j}) + 2\beta(w_{j})} \\
             & = & - \frac{c_{6}}{\rho^{2 + 2\frac{1 + \beta(w_{j})}{I(w_{j})}}}
\end{eqnarray*}

Thus, we have

\begin{eqnarray}
 0 \geq \kappa(z) \geq - \frac{c_{6}}{\rho^{2 + \frac{1 + 
\beta(w_{j})}{I(w_{j})}}} \ , \ para \ |z| < r_{j}
\end{eqnarray}

Take $c_{0} \geq max_{j} \{ c_{6} \}$, and define

\begin{eqnarray*}
 \epsilon = inf \{2 \frac{1 + \beta(w_{j}}{I(w_{j})} | 1 \leq j \leq m \}
\end{eqnarray*}

and

\begin{eqnarray*}
 \kappa_{o}  = sup_{\rho(p)} |\kappa(z)|.
\end{eqnarray*}

Take a \textit{cut off function} $k: [0, \infty) \longrightarrow [0, \infty)$ 
of class $C^{\infty}$ such that $k(s) \geq |\kappa(z)|$ for all
$\rho(p) \leq R_{o}$ and $k(s) = \frac{c_{o}}{s^{2 + \epsilon}}$ , for all $s 
\geq R_{o}$.

So we have

\begin{eqnarray*}
 \int_{0}^{\infty} sk(s)ds & = & \int_{0}^{R_{o}} sk(s)ds + 
\int_{R_{o}}^{\infty} sk(s)ds \\
                           & \leq & \int_{0}^{R_{o}} s\kappa_{o} ds + 
\int_{R_{o}}^{\infty} s \frac{c_{o}}{s^{2 + \epsilon}} ds \\
                           & = & \kappa_{o}R_{o}^{2} + \int_{R_{o}}^{\infty} 
\frac{c_{o}}{s^{1 + \epsilon}} ds \\
                           & = & C_{7} + c_{o} \lim\limits_{y \rightarrow 
\infty} \int_{R_o}^{y} \frac{1}{s^{1 + \epsilon}} ds \\
                           & = & C_{7} -\frac{c_{o}}{\epsilon}  \lim\limits_{y 
\rightarrow 
\infty} \left( \frac{1}{y^{\epsilon}} - \frac{1}{R_{o}^{\epsilon}} \right) \\
                           & < & \infty.
\end{eqnarray*}
\end{proof}
\begin{lem}
 Let $M$ be a complete surface with curvature $\kappa,\ 0\geq \kappa\geq k,$ 
where $k(r)$ is the function of last theorem. If $B_r$ is the geodesic ball of 
$M$ with fixed center then there are constant $c$ and $r_0>0$ such that 
$\text{vol}(B_r)\leq cr^2$ for all $r\geq r_0.$
\end{lem}
\begin{proof}
Let $M_o$ be a model with metric $dr^2+h^2dt^2$ where $h''/h=k$  the function 
$k(s)$ defined in theorem \ref{thm5.1} item (ii). By lemma 4.5 of \cite{kn:G-W}
there is a constant $c_0>1$ such that $r\leq h(r)\leq c_0r$. 

Set $\exp=\exp_{p_0}\colon\R^2\to M,\ \R^2=T_{p_0}M.$ Let $F$ be the set 
$M\setminus \mathcal{C}(p_0)$ where $\mathcal{C}(p_0)$ is the cut locus of 
$p_0.$ Let $U$ be the closure of the connected component of $\exp^{-1}(F)$ 
passing at the origin. The metric into $U$ endowed by $\exp$ has the expression 
$dr^2+g^2dt^2.$  If $B_r$ is the geodesic ball with 
center $p_0$ then 
\[
 \text{vol}_M(B_r)=\int_{U\cap\{|z|\leq r\}}dM,
\]
By the Rauch comparison theorem we have $g(r)\leq h(r)\leq c_0r$ for all $ 
r_0\leq r\leq r^\star,$ where $[r_0,\ r^\star]$ is the interval where the 
radial geodesic $\gamma(r)$ is into $U$. Hence 
\begin{equation}
 \text{vol}_M(B_r)\leq  \text{vol}_M(B_{r_0})+c_1\int_{r_0}^rsds\leq 
c_2r^2\quad r\geq r_0.
\end{equation}

\end{proof}

\section{Proof of theorem \ref{th2}}
 It is 
well know that $\text{vol}(B_r)\leq c_0r^2$ of geodesic balls imply those 
surfaces are parabolic. For example 
\[
 \int_{r_0}^\infty\frac{rdr}{\text{vol}(B_r)}\geq 
\text{const.}\int_{r_0}^r\frac{ds}{s}=\text{const.}\ln(r/r_0),\quad r>r_0,
\]
and by 
\cite{kn:Grig} theorem 11.14 $M$ is parabolic. Then all surfaces of theorem 2 
are parabolic.

\end{document}